\documentclass[12pt, a4paper]{amsart}

\usepackage[left=3cm,top=3.5cm,right=3cm]{geometry}
\usepackage{parskip, color}
\usepackage{xcolor}
\usepackage[normalem]{ulem}
\usepackage{dsfont}

\usepackage{geometry}
\usepackage[T1]{fontenc}
\usepackage{cmbright}
 
 \usepackage{hyperref}
 \usepackage{graphicx}
 \usepackage{tikz}
 \usepackage{wrapfig,float,multicol}
\usepackage{marginnote}

\usepackage{enumerate}

\usepackage{comment}

\newcommand{\R}{\mathbb{R}}
\usepackage{setspace}

\numberwithin{equation}{section}

\newtheorem{theorem}{Theorem}[section]

\newtheorem{lemma}[theorem]{Lemma}

\theoremstyle{definition}

\theoremstyle{remark}

\newcommand{\tu}{\bar{u}}
\newcommand{\ua}{{\bar{u}_a}}
\newcommand{\lambdaa}{{\lambda^{\! ax}_0}}
\newcommand{\lambdan}{\lambda_0^{\! 0}}
\newcommand{\td}{\textrm{d}}
\newcommand{\dx}{\,\textrm{d}x}
\newcommand{\N}{\mathbb{N}}
\DeclareMathOperator{\Ai}{Ai}

\newcommand\abs[1]{\lvert#1\rvert}

\usepackage{fancyhdr}
\newcommand\sbullet[1][.5]{\mathbin{\vcenter{\hbox{\scalebox{#1}{$\bullet$}}}}}
\usepackage{accents}
\newcommand*{\ddt}[1]{%
\accentset{\scriptsize\sbullet}{#1}}

\newcounter{tempcolnum}

\makeatletter
\newcommand{\multicolinterrupt}[1]{
\setcounter{tempcolnum}{\col@number}
\end{multicols}
#1%
\begin{multicols}{\value{tempcolnum}}
}
\makeatother

\def\Xint#1{\mathchoice
 {\XXint\displaystyle\textstyle{#1}}{\XXint\textstyle\scriptstyle{#1}}%
 {\XXint\scriptstyle\scriptscriptstyle{#1}}{\XXint\scriptscriptstyle\scriptscriptstyle{#1}}\!\int}

\def\XXint#1#2#3{{\setbox0=\hbox{$#1{#2#3}{\int}$}\vcenter{\hbox{$#2#3$}}\kern-.5\wd0}}
 \def\dashint{\Xint-}

\definecolor{darkred}{rgb}{0.7,0.1,0.1}

\author{Ben Andrews} \address[Ben Andrews]{Mathematical Sciences
  Institute, Australian National University, ACT 2601 Australia; and
  Yau Mathematical Sciences Center, Tsinghua University, Beijing
  100084, China.}
\email{\href{mailto:Ben.Andrews@anu.edu.au}{\nolinkurl{Ben.Andrews@anu.edu.au}}}
\thanks{The research of the first author was supported by grants
  DP120102462 and FL150100126 of the Australian Research Council.
  The research of the second author was supported by grant FT1301013 of the Australian Research Council.
  }

\author{Julie Clutterbuck} 
\address[Julie Clutterbuck]{School of Mathematics, 
  Monash University,  VIC 3800  Australia}
\email{\href{mailto:Julie.Clutterbuck@monash.edu}{\nolinkurl{Julie.Clutterbuck@monash.edu}}}

\author{Daniel Hauer} \address[Daniel Hauer]{School of Mathematics and
  Statistics, The University of Sydney, NSW 2006, Australia}
\email{\href{mailto:daniel.hauer@sydney.edu.au}{\nolinkurl{daniel.hauer@sydney.edu.au}}}



\begin{document}

\spacing{1.08}

\begin{abstract}{For Schr\"odinger operators on an interval with
    either convex or symmetric single-well potentials, and Robin or
    Neumann boundary conditions, the gap between the two lowest
    eigenvalues is minimised when the potential is constant.  We also have
    results for the $p$-Laplacian.}
\end{abstract}

\title{The fundamental gap for a one-dimensional Schr\"odinger
  operator with Robin boundary conditions}

\maketitle

\section{Introduction}

In studying the eigenvalues of a differential operator, one important
quantity is the gap between the first and second eigenvalues, called
the \emph{fundamental gap}.  This is of both physical and mathematical
importance: in the context of the heat equation, it gives the rate of
collapse of any initial state to the ground state; computationally, it
can control the rate of convergence of a numerical scheme
\cite{saad2011numerical}.

The fundamental gap for the classical Schr\"odinger operator
$-\Delta + V$ has been extensively studied under \emph{Dirichlet}
boundary values.  In one dimension, lower bounds of this gap were
found under various assumptions on $V$
\cite{ashbaugh-benguria-89,MR1948113} until Lavine \cite{MR1185270}
found the sharp result: the gap for a convex potential is minimised by
the gap for a constant potential. The analogous result in higher
dimensions, on a convex domain, was resolved several years later
\cite{fundamental}. Smits considered the question of the lower bounds
on the fundamental gap under \emph{Robin} boundary conditions
\cite{smits1996spectral}, however there are very few results known in
this case.  Laugesen recently studied the Robin eigenvalues, and the gap, on
rectangles \cite{laugesen2019robin}.  The Robin
problem is much more sensitive to the boundary, and is thus more
difficult. For example, the method used in \cite{fundamental} to prove
sharp lower bounds on the Dirichlet fundamental gap uses the property
that the first Dirichlet eigenfunction is \emph{log-concave}: in
recent work, we have shown that the first Robin eigenfunction does not
always enjoy that property \cite{andrews2017non}.

Thus, one important motivation here is to find methods to derive sharp
lower bounds on the gap that do not rely on the
log-concavity of the first eigenfunction.  In one dimension, Lavine's
proof of the lower bound for the gap is such a method.

In this paper, we revisit Lavine's method and establish sharp lower
bounds of the fundamental gap for the classical (linear) Schr\"odinger
operator on a bounded interval under Robin boundary conditions.  We
not only deal with the case where the Robin parameter $\alpha$ is
positive, but also for $-\frac{1}{2}\le \alpha\le0$, thus also
including the Neumann case.  The same statements also hold for the
Dirichlet case (sometimes referred to as $\alpha=\infty$).  We also
extend some results to the nonlinear Schr\"odinger operator associated
with the $p$-Laplace operator.

We can extend the methods to Robin boundary conditions
largely because the boundary conditions generally appear in forms such
as $[u_1 u_0'-u_1' u_0]^{-1}_{1}$ and so they not only vanish in the
Neumann or Dirichlet cases, but also in the Robin case.

We now  introduce our notation and assumptions used through this
paper. Let $I$ denote the open interval $(-1,1)$ and $V$ a potential
function in $C(\bar{I})$. 
The eigenvalue problem for the classical Schr\"odinger operator
$-\frac{\td^{2}}{\dx^{2}}+V$ on $I$ is to find eigenpairs
$(u_i,\lambda^{\! V}_i)$ solving
\begin{equation}
  \label{main}
 - u''+ V u= \lambda^{\!\! V} u \quad \text{ on }I
 \end{equation}
 subject to either homogeneous \emph{Robin or Neumann boundary conditions}
 \begin{equation}
   \label{main boundary} 
   u'(\pm 1)=\mp \alpha u(\pm 1), 
 \end{equation}
or homogeneous \emph{Dirichlet boundary condition} 
\begin{equation} 
  \label{Dirichlet bc}
  u(\pm 1)=0.
\end{equation}
The boundary conditions~\eqref{main boundary} are called
  \emph{Neumann} if $\alpha=0$, and otherwise \emph{Robin} with
  \emph{Robin parameter} $\alpha\in \R$. Dividing~\eqref{main
    boundary} by $\alpha>0$ and sending $\alpha\to +\infty$, one
  recovers the \emph{Dirichlet} conditions~\eqref{Dirichlet bc}.

  For $1<p<+\infty$, we also consider the Robin eigenvalue problem for
  the non\-linear Schr\"odinger operator $-\Delta_{p}+V\abs{\cdot}^{p-2}\cdot$ associated
  with the $p$-Laplace operator
  $\Delta_{p}u:=\left(|u'|^{p-2} u'\right)'$. Here we seek to find
  eigenpairs $(u_i,\lambda^{\! V}_i)$ solving
\begin{equation} 
  \label{p laplace 2}
-\Delta_{p}u+ V|u|^{p-2}u=\lambda^{\!\! V}|u|^{p-2}u
\quad \text{ on $I$,}
\end{equation}
with Robin boundary conditions
\begin{equation}
  \label{p laplace bc}
  |u'|^{p-2}u'=\mp \alpha |u|^{p-2}u \text{ at $x=\pm 1$. } 
\end{equation}
The Dirichlet boundary condition is again \eqref{Dirichlet bc}.

For
$p=2$ equations~\eqref{p laplace 2} and~\eqref{p laplace bc} reduce to
the classical linear ones~\eqref{main} and~\eqref{main boundary}. For
given $V\in L^{1}(I)$, $\lambda\in \R$, and $\alpha\in \R$, we
call a function $u\in C(\bar{I})$ a \emph{solution} of
equation~\eqref{p laplace 2} satisfying~\eqref{p laplace bc} provided
$u\in W^{1,p}(I)$ and satisfies
\begin{displaymath}
  \int_{I}|u'|^{p-2}u'\xi'+ \big(V-\lambda\big) |u|^{p-2}u\,\xi\,\dx
  +\alpha \Big[(|u|^{p-2}u\,\xi)_{\vert x=1}+(|u|^{p-2}u\,\xi)_{\vert x=-1}\Big]=0
\end{displaymath}
for every $\xi \in C^{1}(\bar{I})$. Since for every solution $u$
of~\eqref{p laplace 2}, $\lambda^{\!\!
  V}-V\,|u|^{p-2}u\in L^{1}(I)$, one has that $|u'|^{p-2}u'\in
W^{1,1}(I)$ and since $W^{1,1}(I)$ is continuously embedded into
$C(\bar{I})$, one finds that $u'\in C(\bar{I})$. Thus, every solution $u$ of~\eqref{p
  laplace 2} has regularity $u\in C^1(\bar{I})$.

The
Ljusternik-Schnirelmann theory~\cite{MR0467421,le2006eigenvalue} ensures the existence
of a
sequence $(\lambda^{\!\!
  V}_{i},u_{i})_{i\ge 0}$  of eigenpairs $(\lambda^{\!\!
  V}_{i},u_{i})$ for the nonlinear Schr\"odinger operator
$-\Delta_{p}+V\abs{\cdot}^{p-2}\cdot$ with Robin boundary conditions~\eqref{p laplace bc}.  For every $i\in \N$,
\begin{equation}
  \label{eq:1}
  \lambda^{\!\!V}_{i}:=\inf_{W\in \Gamma_{i}}\max_{u\in W}\mathcal{R}[u,V,p,\alpha],
\end{equation}
where $\mathcal{R}$ is the \emph{Rayleigh
  quotient}
\begin{equation}
  \label{p-Rayleigh}
  \mathcal{R}[u,V,p,\alpha]=\frac{\displaystyle\int_I |u'|^p + V |u|^p \,\dx + 
    \alpha\left[ |u|^p(1)+|u|^p(-1)\right]}{\displaystyle \int_I |u|^p\,\dx}.
\end{equation}
In~\eqref{eq:1}, the maximum is attained among all $u\in W\subseteq \Gamma_i$, where   $\Gamma_i$ is a specific closed subset of 
  $W^{1,p}(I)$.  To obtain 
eigenpairs $(\lambda^{\!\!
 V}_{i},u_{i})$ for homogeneous Neumann
boundary conditions, one chooses $\alpha=0$, 
 while for homogeneous  
Dirichlet boundary conditions~\eqref{Dirichlet bc}, one
needs to replace the space $W^{1,p}(I)$ by $W^{1,p}_{0}(I)$.

Our primary object of interest in this paper is the fundamental gap 
\begin{displaymath}
  \Gamma_{\!p}(V):=\lambda^{\!\! V}_1-\lambda^{\!\! V}_0.
\end{displaymath}
We 
write $\Gamma_{\!p}(0)$  for the fundamental gap of
the zero potential $V\equiv 0$. A potential $V$ 
is called  \emph{single-well} if there is a point
$x_{0}\in I$ such that $V$ is non-increasing along $(-1,x_{0})$ and
non-decreasing along $(x_{0},1)$. A potential
$V$ on $I$ is called \emph{symmetric} if $V(-x)=V(x)$.

Our first result concerns the fundamental gap for such symmetric, single-well potentials.  
The analogous theorem with Dirichlet
boundary conditions, and  $p=2$, is due to Ashbaugh and
Benguria \cite{ashbaugh-benguria-89}.

\begin{theorem}\label{single well p}
  For $\alpha\in \R$, 
  consider the nonlinear eigenvalue problem~\eqref{p laplace 2} with Robin boundary
  conditions~\eqref{p laplace bc}. Then for every symmetric, single
  well potential $V$, the fundamental gap satisfies
  \begin{displaymath}
    \Gamma_{\!p}(V)\ge \Gamma_{\!p}(0)
  \end{displaymath}
with equality only when $V$ is constant.   
\end{theorem}

Our next theorem shows that the fundamental gap $\Gamma_{\!p}(V)$ with
convex potential $V$ is minimised by a linear potential.

\begin{theorem} \label{linear better than convex p} For $\alpha\in \R$,
    consider the nonlinear eigenvalue problem
  \eqref{p laplace 2} with Robin boundary conditions \eqref{p laplace
    bc}. Then for every convex potential $V$ which is
  not affine, there exists a linear potential $V_{a}=ax$, 
  such that 
\begin{displaymath}
  \Gamma_{\!p}(V)>\Gamma_{\!p}(ax).
\end{displaymath} 
\end{theorem}

The last two theorems relate to the classical, $p=2$, case only.

\begin{theorem} 
\label{constant vs linear}   
Consider the linear eigenvalue problem~\eqref{main} with Robin
boundary conditions \eqref{p laplace bc} with  $\alpha\ge -\frac{1}{2}$, 
and let $V_a=ax$ be a linear
potential, $a\in \R$. Then the fundamental gap is bounded below by the gap for a zero potential, 
\begin{displaymath}
  \Gamma_{\!2}(ax)\ge \Gamma_{\!2}(0).
\end{displaymath} 
with equality only when $a=0$.
\end{theorem}

{Combining Theorem~\ref{linear better than convex p} with
Theorem~\ref{constant vs linear}, we immediately have:}

\begin{theorem} \label{convex potential} Consider the linear
  eigenvalue problem~\eqref{main} where $V$ is a \emph{convex
    potential}, and with Robin boundary conditions \eqref{main
    boundary} with $\alpha\ge -\frac{1}{2}$.  Then the fundamental gap
  satisfies
  \begin{displaymath}
  \Gamma_{\!2}(V)\ge \Gamma_{\!2}(0).
\end{displaymath} 
\end{theorem}

In the linear case $p=2$ and under homogeneous Dirichlet and Neumann
boundary conditions, Theorems \ref{linear better than convex p},
\ref{constant vs linear} and \ref{convex potential} are due to Lavine
\cite{MR1185270}.

The structure of this paper is as follows. In Section \ref{technical
  results} we establish some necessary technical results, particularly
about the shape of the first two eigenfunctions.  In Section
\ref{single well section}, we prove Theorem \ref{single well p}, for
symmetric single-well potentials.  In Section \ref{section four}, we
prove Theorem \ref{linear better than convex p}.  In Section
\ref{linear vs convex section}, we prove some further technical
results to support the argument in Section \ref{section six}, where we
prove Theorem \ref{constant vs linear}: this proof is only for the
\emph{classical} Schr\"odinger operator, with $p=2$.

\section{Some {preliminary} results}\label{technical results}
We normalise
all eigenfunctions  so that $\int_{I} |u_i|^p\dx=1$ and
$u_{i}>0$ on $(-1,-1+\epsilon)$ for some small
$\epsilon>0$. If $(u_{0},\lambda_{0}^{\!\! V})$ is
the first eigenpair then $u_{0}$ minimises the Rayleigh
quotient~\eqref{p-Rayleigh}. Since also $|u_{0}|$ is a minimiser
of~\eqref{p-Rayleigh}, we can choose $u_{0}\ge 0$ on
$\overline{I}$. But as we assume  that the potential $V$
is bounded from below on $I$, and since $u_{0}$ is a solution of~\eqref{p
  laplace 2}, the strong maximum principle
(see~\cite[Theorem~5.3.1]{MR2356201}) implies that $u_{0}>0$ on $I$.

We begin with a Hellmann--Feynmann result for the variation of eigenvalues
with respect to a family of potentials. 

\begin{lemma} \label{varying potential} 
  Let $\lbrace V^{t} \rbrace_{t\in J}$ be a family of potentials
  $V^{t}\in L^{1}(I)$ varying differentiably in the parameter
  $t\in \R$. Then for every
  eigenpair $(u_{i},\lambda_{i}^{\!\! V})$ of the nonlinear eigenvalue problem
  \eqref{p laplace 2} with Robin boundary conditions \eqref{p laplace
    bc},
\begin{equation*}
  \label{lambda dot}
  \frac{\partial}{\partial t}\lambda_i^{\!\! V^{t}}=\int_{I} \ddt{V}^{t} |u_i|^p\,\dx     
\end{equation*}
where $\ddt{V}^{t}$ indicates the derivative
$\frac{\partial}{\partial t}V^{t}$ with respect to $t$. In particular,
for the fundamental gap,
\begin{equation} 
  \label{gap dot}
  \frac{\partial}{\partial t}\Gamma_{\!p}(V^{t})=\int_{I} \ddt{V}^{t}
  \left(|u_1|^p-|u_0|^p\right)\dx.   
\end{equation}
\end{lemma}
\begin{proof}
  By using the Rayleigh quotient \eqref{p-Rayleigh} and the fact that
  $(u_{i},\lambda_{i}^{\!\! V})$ is an eigenpair of~\eqref{p laplace
    2} satisfying Robin boundary conditions \eqref{p laplace bc}, \allowdisplaybreaks
\begin{align*}
\frac{\partial}{\partial t}\lambda_i^{\! V^{t}}&= \frac{1}{\int_I |u_{i}|^p \dx} \Bigg[\int_I
                            p|u'_{i}|^{p-2}u'_{i} \ddt{u}'_{i} +
                                                 \ddt{V}^{t}|u_{i}|^p
                                                 + V^{t} p |u_{i}|^{p-2}u_{i} \ddt{u}_{i}\,\dx
                            \Bigg.\\
&\hspace{4cm}\Bigg. 
+ \alpha\Big[ p |u_{i}|^{p-2} u_{i} \ddt{u}_{i}(-1)+ p |u_{i}|^{p-2} u_{i} \ddt{u}_{i}(1)\Big]  \Bigg]\\
& \hspace{1cm} -\frac{1}{\left(\int_I |u_{i}|^p \dx\right)^2} \left[
\int_I  |u'_{i}|^p+ V^{t}|u_{i}|^p \dx\right.\\
&\hspace{4.5cm}\Bigg.+  \alpha\Big[|u_{i}|^p(-1)+|u_{i}|^p(1)\Big] \Bigg]
  \left[\int_I
  p|u_{i}|^{p-2}u_{i}\ddt{u}_{i} \dx 
  \right] \\
&= \frac{1}{\int_I |u_{i}|^p \dx} 
\left[ \lambda^{\!\! V^{t}}_{i}\!\int_I
  p|u_{i}|^{p-2}u_{i}\ddt{u}_{i} \dx +\int_{I} \ddt{V}^t|u_{i}|^p \dx
  \right]\\ 
&\hspace{5cm}  -\frac{\lambda^{\!\! V^{t}}_{i}}{\int_I |u_{i}|^p \dx} \int_I
  p|u_{i}|^{p-2}u_{i}\ddt{u}_{i}\dx\\
&=  \frac{1}{\int_I |u_{i}|^p \dx} \int_I   \ddt{V}^{t}|u_{i}|^p \dx. 
\end{align*}
\end{proof}

For a given $\lambda\in \R$, let  $u$ be a solution of the nonlinear eigenvalue problem~\eqref{p
  laplace 2}-\eqref{p laplace bc}. Let $v$ be defined by
\begin{equation}
  \label{eq:5}
  v=\frac{|u'|^{p-2}u'}{|u|^{p-2}u}\qquad\text{on $I$.}
\end{equation}
Then $v$ is a solution of the Riccati equation
\begin{equation}
  \label{q prime}
  v'=(V-\lambda)- (p-1)\, |v|^{\frac{p}{p-1}}
\end{equation}
on $I$, and by~\eqref{p laplace bc},
satisfies the inhomogeneous Dirichlet boundary conditions
\begin{equation}
  \label{eq:2}
  v(\pm 1)=\mp \alpha.
\end{equation}
For every $\lambda\in \R$, the function
\begin{equation*}
  \label{eq:6}
  f_{\lambda}(x,v):=(V-\lambda)-
  (p-1)\,|v|^{\frac{p}{p-1}}\qquad\text{for a.e. $x\in \bar{I}$ and
    all $v\in \R$,}
\end{equation*}
has a continuous partial derivative
$\frac{\partial}{\partial v}f_{\lambda}(x,v)=-p
|v|^{\frac{2-p}{p-1}}v$, uniformly for a.e. $x\in I$. Thus
Gronwall's lemma implies that for every $c\in \R$ and
$x_{0}\in \bar{I}$, there can be at most one bounded solution $v$ on
$\overline{I}$ of
~\eqref{q prime} 
satisfying the initial condition $v(x_{0})=c$. 

The first eigenfunction is strictly positive 
and so the corresponding solution $v_{0}$ of 
~\eqref{q
  prime} is bounded on $\bar{I}$ and hence  is
unique.
 By~\eqref{eq:5}, this means 
the first eigenpair $(u_{0},\lambda_0^{\!\! V})$ is
\emph{simple}, in the sense that 
any two solutions 
of~\eqref{p laplace 2}-\eqref{p laplace bc}
for the same eigenvalue  are linearly dependent.

Concerning the simplicity of the other eigenpairs,
a generalisation of \emph{Courant's nodal domain theorem} for the 
$p$-Laplace operator  with homogeneous Dirichlet
boundary conditions on a bounded smooth domain in $\R^{d}$ was
obtained by Dr\'{a}bek and Robinson~\cite{MR1900460}. 
A \emph{nodal domain} is defined as a maximal
connected open subset $\{
u(x)\neq
0\}$. 
In one dimension, a Sturm-Liouville theory for nonlinear
Schr\"odinger operators of the form $-\Delta_{p}+V\abs{\cdot}^{p-2}\cdot$ on the bounded
interval $(0,b)$ satisfying $u'(0)=0$ and Robin boundary conditions at the right
endpoint $x=b$ was elaborated by several authors 
(e.g. ~\cite{MR1120904} or~\cite[Theorem~5 \& subsequent
Corollary]{MR1605441}). The results and
techniques of the last two references imply that a nodal
domain theorem for $-\Delta_{p}+V\abs{\cdot}^{p-2}\cdot$ with Robin
boundary condition holds. We omit the details.

\begin{lemma}[Nodal domain theorem for Robin boundary conditions]
  \label{lem:nodal-domain}
  Let $V\in C(\bar{I})$ and $\alpha\in \R$. Then each
  eigenvalue  $\lambda_i^{\!\! V}$ of the Schr\"odinger
  operator $-\Delta_{p}+V\abs{\cdot}^{p-2}\cdot$ with Robin boundary conditions~\eqref{p
    laplace bc} is simple, satisfies
  \begin{displaymath}
  \lambda_0^{\!\! V}<\lambda_1^{\!\! V}< \lambda^{\!\! V}_2< \cdots \rightarrow \infty.
\end{displaymath}
Moreover, the corresponding eigenvector
$u_{i}$ has exactly $i+1$ nodal domains in $I$.
 \end{lemma}

With the help of the preceding lemma, we obtain the following
monotonicity property.

\begin{lemma} \label{monotony p} 
  For $\alpha\in \R$, let
  $(u_{0}, \lambda_0^{\!\! V})$ and $(u_{1},\lambda_1^{\!\! V})$ be
  the first and second eigenpair of the Schr\"odinger operator
  $-\Delta_{p}+V\abs{\cdot}^{p-2}\cdot$ with Robin boundary conditions~\eqref{p laplace bc}.
  Then the ratio $u_1/u_0$ is monotonically decreasing in $x$.
\end{lemma}

\begin{proof}
  By the nodal domain theorem (Lemma~\ref{lem:nodal-domain}), $u_{1}$
  admits exactly one zero $x_0$ in $I$, and by construction, $u_1$ is positive
near $x=-1$. Thus 
\begin{displaymath}
  \phi:=\log \frac{u_1}{u_0} \end{displaymath}
 is well-defined on $[-1,x_0)$. We claim that $\phi'<0$ on $[-1,x_0)$. We
compute 
\begin{equation*} 
\phi'=  \frac{u_1'}{u_1}-\frac{u_0'}{u_0}= |v_1|^{p^{\star}}v_1-
|v_0|^{p^{\star}}v_0,
\end{equation*}
where $v_{i}$ is defined by~\eqref{eq:5} for $u=u_{i}$ and $p^{\star}= -(p-2)/(p-1)$. 
As $s\mapsto s|s|^{p^\star}$ is increasing, $\phi'<0$ on $[-1,x_0)$ if $v_1<v_0$ on $[-1,x_0)$. 

 By~\eqref{eq:2},
$v_{0}(-1)=\alpha=v_{1}(-1)$ and since
$\lambda_0^{\!\! V}<\lambda_1^{\!\! V}$, the
Riccati equation~\eqref{q prime} implies that
\begin{displaymath}
  v'_{1}(-1)=(V-\lambda_1^{\!\! V})-(p-1)|\alpha|^{\frac{p}{p-1}}
  <(V-\lambda_0^{\!\! V})-(p-1)|\alpha|^{\frac{p}{p-1}}=v'_{0}(-1).
\end{displaymath}
Hence  $v_1<v_0$ in a neighbourhood of $-1$.

Since $v_i\in C(\bar{I})$, there 
exists 
 a largest $y_{0}\in (-1,x_{0}]$ such that
$v_{1}<v_{0}$ on $(-1,y_{0})$. 

If $y_{0}<x_{0}$, then
$v_1(y_0)=v_0(y_0)$.   
The Riccati equation~\eqref{q prime}, this time evaluated at $y_{0}$, implies that 
 $v'_1(\xi_0)<v'_0(\xi_0)$ and
hence $v_{1}>v_{0}$ on $(y_{0}-\epsilon, y_0)$,
contradicting that $v_1<v_0$ on $(-1,y_0)$. Therefore $y_{0}=x_{0}$,
$v_1<v_0$,  and  $\phi'<0$ on $[-1,x_0)$.

Similarly, one can show that $\log(-u_1/u_0)$ is increasing on
$(x_0,1]$.  It follows that the ratio $u_1/u_0$ is
monotonically decreasing on the whole interval $\bar{I}$.
\end{proof}

\begin{lemma} \label{lemma 3.1 p} 
 {Let
  $(u_{0}, \lambda_0^{\!\! V})$ and $(u_{1},\lambda_1^{\!\! V})$ be
  the first and second eigenpair of the Schr\"odinger operator
  $-\Delta_{p}+V\abs{\cdot}^{p-2}\cdot$ with Robin boundary conditions~\eqref{p laplace
    bc}.  Then
  $|u_1|^p-|u_0|^p$ has at least one and at most two zeroes in $I$.}	
To be precise, there exists
  $\xi_{-}$, $\xi_{+}\in \bar{I}$ at least one of which is an interior point of $I$,
   such that $\xi_{-}<\xi_{+}$,
   and 
  \begin{equation}
    \label{u1-u0 p}
    \begin{split}
      |u_1|^p-|u_0|^p<0&\text{ on
        $(\xi_{-},\xi_{+})$}\quad\text{and}\\
      |u_1|^p-|u_0|^p>0&\text{ on
        $I\setminus [\xi_{-},\xi_{+}]$}.
    \end{split}
  \end{equation}

\end{lemma}

\begin{proof}   

 The normalised eigenfunctions satisfy $\int_{I}|u_0|^p\,
  \dx=\int_{I}|u_1|^p\,\dx=1$,  and so 
  \begin{displaymath}
    \psi:=|u_1|^p-|u_0|^p 
  \end{displaymath}
  has mean value $\dashint_{I}\psi\,\dx=0$. At the zero
  $x_{0}$ of $u_{1}$,  $\psi(x_{0})=-|u_0|^p(x_0)<0$, and so
  there must  be a $y_0\in I $ such that
  $\psi(y_{0})>0$.    Since $\psi$ changes sign in $I$, it
 has at least one zero $\xi_{0}$.  
 On the other hand, $\xi$ is a zero of $\psi$ if and only if the ratio
  $\frac{u_{1}}{u_{0}}(\xi)=\pm 1$.    By Lemma \ref{monotony p}, 
  $u_1/u_0$ is monotonically decreasing, so
  there exists at most two points $-1\le \xi_{-}<\xi_{+}\le 1$
  satisfying $\frac{u_1}{u_0}(\xi_{-})=1$ and
  $\frac{u_1}{u_0}(\xi_{+})=-1$. One of these is $\xi_{0}$.
  Therefore, $\psi$ has at least one zero in $I$
  and at most two zeroes in $\bar{I}$.
\end{proof}

\section{Fundamental gap estimates for symmetric single-well potentials}

\label{single well section}

In this section, we prove Theorem \ref{single well p} using  Lemma~\ref{varying potential}--Lemma~\ref{lemma 3.1 p}.

\begin{proof}[Proof of Theorem \ref{single well p}]
  Let $V$ be a symmetric single-well potential. Define  $\{V^{t}\}_{t \in \R^+}$  by $V^{t}(x):=t V(x)$.
  For this family of potentials, 
  $\ddt{V}^t=V$ and so  \eqref{gap dot} 
  results in
  \begin{displaymath}
    \frac{\partial}{\partial t}\Gamma_{\!p}(V^{t})=\int_{I} V
  \left(|u_1|^p-|u_0|^p\right)\dx
  \end{displaymath}
  where $(u_{0}, \lambda_0^{\!\! V})$ and $(u_{1},\lambda_1^{\!\! V})$
  are the first and second eigenpairs of the Schr\"odinger operator
  $-\Delta_{p}+V\abs{\cdot}^{p-2}\cdot$ with Robin boundary
  conditions~\eqref{p laplace bc}.

  As $V$ is symmetric on $\bar{I}$, the eigenfunctions $u_{0}$ and
  $u_{1}$ are symmetric and antisymmetric respectively.  Hence
  $|u_1|^p-|u_0|^p$ is symmetric, and the two zeroes of
  $|u_1|^p-|u_0|^p$ found in Lemma~\ref{lemma 3.1 p} are likewise
  symmetric, with $\xi_{-}=-\xi_+$.
  
  Since
  $V$ is symmetric and single-well, $V$ is non-increasing on
  $(-1,0)$ and non-decreasing on
  $(0,1)$.    Thus, using~\eqref{u1-u0 p}, we have 
\begin{align*}
  \frac{\partial}{\partial t}\Gamma_{\!p}(V^{t})&= \int_{-1}^1 V\left(|u_1|^p-|u_0|^p\right)\,\dx\\
&=  \int_{-1}^{\xi_{-}} V\left(|u_1|^p-|u_0|^p\right)\,\dx
  +\int_{\xi_{-}}^{\xi_{+}} V\left(|u_1|^p-|u_0|^p\right)\,\dx\\
&\hspace{4cm}  +\int_{\xi_{+}}^1V\left(|u_1|^p-|u_0|^p\right)\,\dx\\
& \ge
  V(\xi_{-})\int_{-1}^{\xi_{-}}\left(|u_1|^p-|u_0|^p\right)\,\dx 
+V(\xi_{\pm})\int_{\xi_{-}}^{\xi_{+}} \left(|u_1|^p-|u_0|^p\right)\,\dx\\
&\hspace{4cm} +V(\xi_{+})\int_{\xi_{+}}^1\left(|u_1|^p-|u_0|^p\right)\,\dx\\
&=  V(\xi_{\pm}) \, \int_{I} \left(|u_1|^p-|u_0|^p\right)\dx  =0
\end{align*}
with equality only when $V$ is constant. Summarising, we have shown that
\begin{displaymath}
  \frac{\partial}{\partial t}\Gamma_{\!p}(V^{t})\ge 0 
\end{displaymath}
with equality only when $V$ is constant. Integrating 
with respect to $t$ over $(0,1)$ proves the theorem.
\end{proof}

\section{Comparison of the fundamental gap between convex and
    linear potentials}  \label{section four}

{Using a similar strategy to that used in the previous section, we can show that
the fundamental gap among convex potentials is minimised by a
linear one.} 

\begin{proof}[Proof of Theorem \ref{linear better than convex p}]
  Let $V$ be a convex potential on $I$ which is not affine.  Let
  $\xi_{-}$ and $\xi_+\in \bar{I}$ be such that \eqref{u1-u0 p} holds,
  for the corresponding eigenfunctions $u_0$, $u_1$.
 
  Let $L_V(x)=a x+b$
    be the line that intersects the graph of $V$ at $\xi_{-}$
  and $\xi_+\in \bar{I}$. By the convexity of $V$,
  \begin{displaymath}
    V-L_V\ge 0\text{ on }(-1,\xi_{-}),\quad
    V-L_V\le 0\text{ on }(\xi_{-},\xi_{+}),\quad
    V-L_V\ge 0\text{ on }(\xi_{+}, 1).  
  \end{displaymath}
  Consequently, 
  using \eqref{u1-u0 p},
\begin{displaymath}  
(V-L_V)(|u_1|^p-|u_0|^p)\ge 0 \text{ on }\bar{I}.
\end{displaymath}
In particular, since  $V$ is not affine,
 there exists a
set 
of
positive
 measure 
 on which the
last inequality is strictly positive and hence
\begin{displaymath}
  \int_{I} (V-L_V)(|u_1|^p-|u_0|^p)\,\dx>0.
\end{displaymath}
For the family $\{V^{t}\}_{t\in [0,1]}$ given by $V^{t}= t V+
(1-t)L_V$,
 $\ddt{V}^t=V-L_{V}$ and
so \eqref{gap dot} shows that
\begin{displaymath}
    \frac{\partial}{\partial t}\Gamma_{\!p}(V^{t})=\int_{I} \big(V-L_{V}\big)
  \left(|u_1|^p-|u_0|^p\right)\dx>0.   
  \end{displaymath}
Integrating this inequality with respect to $t$ over $(0,1)$ gives
\begin{displaymath}
    \Gamma_{\!p}(V)>\Gamma_{\!p}(L_{V})=\Gamma_{\!p}(a x),
  \end{displaymath}
  where $ax$ is the purely linear part of $L_V$.  We can drop the
  constant term $b$ because adding a constant to the potential shifts
  all eigenvalues by that constant, and therefore has no effect on the
  gap.
\end{proof}

\section{Further technicalities}
\label{linear vs convex section}

In this section we derive some  technical results, mostly for linear potentials, which are necessary for
proving Theorem \ref{convex potential}.

\begin{lemma}   
  Suppose that $\int x(u_1^2-u_0^2)=0$.    Then $u_1^2-u_0^2$ has exactly two interior zeroes, and 
\begin{equation} 
 \label{two zeroes} 
  u_1(1)^2-u_0(1)^2 >0\text{ and }u_1(-1)^2-u_0(-1)^2 >0.
\end{equation}
\end{lemma}

\begin{proof}

  From Lemma \ref{lemma 3.1 p}, at least one of these is positive, and
  $u_1^2-u_0^2$ has either one or two interior zeroes.  Suppose that
  $\xi_0$ is the \emph{sole} zero of $u_1^2-u_0^2$, so that
  $(x-\xi_0)(u_1^2-u_0^2)$ is nonzero and has the same sign for all
  $x\not=\xi_0$.  Then
  \begin{displaymath}
   0\not=\int (x-\xi_0)(u_1^2-u_0^2)\dx= \int x(u_1^2-u_0^2)\dx-
    \xi_0 \int u_1^2-u_0^2\dx= 0,
  \end{displaymath}
  where in the last step we use that
  $\int u_1^2\dx=\int u_0^2\dx$. The contradiction implies that
  $u_1^2-u_0^2$ has two zeros.
\end{proof}

Now we  compare the first eigenfunctions in the cases that $V$
is linear and $V$ is zero.

\begin{lemma} \label{compare to zero} 
  For $a\ge 0$, let $(u^{ax}_0,\lambdaa)$ be the first eigenpair of the Schr\"odinger
  operator $-\Delta+ax$ with Robin boundary conditions~\eqref{main
    boundary}. Then for $a>0$, the ratio $u^{ax}_0/u^0_0$ is monotone
  decreasing along $\bar{I}$.
\end{lemma}

\begin{proof}

Let  $x_1=(\lambdaa-\lambdan)/a$, for $a>0$. We claim that  $x_1\in I$.

For later convenience, we work with any eigenpair
$(u^{ax},\lambda^{ax})$.  We can use Lemma \ref{varying potential} to
write
\begin{equation} \label{wellography}
 \lambda^{ax}-\lambda^0
 = \int_{t=0}^{a} \frac{d}{dr} \lambda^{rx} \,dr
 =  \int_{t=0}^{a} \int_{I} x\left(u^{rx}\right)^2 \dx \,dr. 
 \end{equation}
However, on $I$, $-\left(u^{rx}\right)^2< x\left(u^{rx}\right)^2<  \left(u^{rx}\right)^2$, and so
$$ -1= - \int_{I} \left(u^{rx}\right)^2 \dx <  \int_{I} x\left(u^{rx}\right)^2 \dx< \int_{I} \left(u^{rx}\right)^2 \dx=1.$$
Using this estimate for the  integrand of \eqref{wellography} leads to 
\begin{equation}-{a} <\lambda^{ax}-\lambda^0 <a, \label{x1 in I}
\end{equation}  and therefore $x_1
\in (-1,1)$.

We now write the first eigenfunctions for linear potentials as 
$u^{ax}_0=\ua$ and $u^0_0=\tu$. For any  $x\in (-1, x_1]$, 
\begin{align*}
\left(\frac{\ua}{\tu}\right)'(x)
  &=  \frac1{\abs{\tu}^2}\left( \ua'\tu -\ua{\tu}'\right)(x)\\
&=  \frac1{{\tu}^2}\int_{-1}^x \left(\ua'\tu -\ua{\tu}'\right)'\,dy 
  \text{   where we use the boundary values at $-1$}\\
&=  \frac1{{\tu}^2}\int_{-1}^x \ua''\tu+ \ua'{\tu}' -\ua'{\tu}'-\ua{\tu}''\,dy \\
&=  \frac1{{\tu}^2}\int_{-1}^x      (ay-\lambdaa+\lambdan)\ua\tu \,dy \\
&<0,
\end{align*}
where we have used that $ay<ax<ax_1=\lambdaa-\lambdan$.  Similarly, for $x \in(x_1,1)$,
\begin{align*}
\left(\frac{\ua}{\tu }\right)' (x)  =
 - \frac1{{\tu}^2}\int_{x}^1 \left({\ua}'\tu -\ua{\tu}'\right)'\dx 
=  -\frac1{{\tu}^2}\int_{x}^1      (ay-\lambdaa+\lambdan)\ua\tu \dx 
<0.
\end{align*}
\end{proof}

\begin{lemma}  \label{lemma 5.4} 
For $a\ge 0$, let $(u^{ax}_0,\lambda_0^{ax})$ be the first eigenpair of the
Schr\"odinger operator $-\Delta+ax$ with Robin boundary conditions~\eqref{main boundary}.  
Then for $a>0$, 
\begin{displaymath}
\abs{u^{ax}_0}^2(1)-\abs{u^{ax}_0}^2(-1)<0.
\end{displaymath}
\end{lemma}

\begin{proof}   
  Again we write $u^{ax}_0=\ua$ and $u^0_0=\tu$. 
  From Lemma \ref{compare to zero}, $\ua/\tu$ is decreasing, so that
  for $x>0$,
\begin{displaymath} 
  \frac{\ua(-x)}{\tu(-x)}> \frac{\ua(x)}{\tu(x)}  
\end{displaymath}
however $\tu$ is even, so 
\begin{equation*}\label{eq:monotonicity}
  {{\ua}^2(-x)}> {\ua^2(x)}; 
\end{equation*}
then the result follows directly with $x=1$.
\end{proof}

\begin{lemma}  \label{lemma 5.3} 
  For $a>0$, let $\lambda^{ax}_1$ be the second eigenvalue 
 of the Schr\"odinger operator $-\Delta+ax$ for Robin
 boundary conditions~\eqref{main boundary} with parameter  $\alpha$. Then
\begin{equation} \label{a alpha eqn} 
  \alpha^2+ \lambda_1^{ax}+a>0.
\end{equation}
\end{lemma}

\begin{proof}
From \eqref{x1 in I}, we have 
$ \lambda_1^{ax}+a \ge \lambda_1^0$. Then
 \begin{displaymath}
     \alpha^2+\lambda_1^{ax}+a\ge \alpha^2+ \lambda_1^{0}.
 \end{displaymath}
 We estimate $\lambda_1^0$: we have a zero potential and everything
 may be done explicitly.

 If $\alpha\ge0$ then \eqref{a alpha eqn} follows directly from the
 positivity of the Rayleigh quotient.

 When $\alpha<0$, $\lambda_0$ is negative, since otherwise the general
 solution to $-u''=\lambda u$ is given by
 $u(x)=c_1\cos\sqrt{{\lambda}}x+ c_2\sin\sqrt{{\lambda}}x$; however
 the boundary condition requires that
 $\sqrt{{\lambda}}\tan \sqrt{{\lambda}}=\alpha$, which cannot be
 satisfied if $\alpha<0$.

 In the case that $-1<\alpha<0$, $\lambda_1$ is positive: to be
 precise, $\lambda_1= \mu^2$, where $\mu$ solves
 $-\alpha\, {\tan\mu}={\mu}$, and $u_1=\sin\mu x$.  If $\alpha=-1$,
 then $\lambda_1=0$, with a linear eigenfunction.  If $\alpha<-1$,
 then $\lambda_1$ is negative, however the claim still holds, since
 then $\lambda_1=-\mu^2$ where $\mu$ solves $\mu=-\alpha \tanh\mu$,
 and hence
 \begin{displaymath}
   \alpha^2+ \lambda_1=\alpha^2-\mu^2= \left(\frac{\mu}{\tanh\mu}\right)^2(1-\tanh\mu^2) >0.
\end{displaymath}
The conclusion follows.
\end{proof}

\begin{lemma}  
  Let $(u,\lambda)$ be an eigenpair of the
  Schr\"odinger operator $-\Delta+ax$. Then
\begin{equation}  \label{eq:g is 1}
  \left[ ({u}')^2+ (\lambda-ax){{u}}^2 \right]_{-1}^{1}= -a
\end{equation} and
\begin{equation}
  \label{eq:g is xx}
  \left[  ({{u}}')^2+ (\lambda-ax+1){{u}}^2  - 2x {{u}} {{u}}' \right]_{-1}^{1}
  = 4\lambda\int_{-1}^1 x {{u}}^2  - 5a \int_{-1}^1 x^2{{u}}^2   \dx.   
\end{equation}
\end{lemma}

\begin{proof}
We calculate 
\begin{align*}
\left[ {u}'^2 + (\lambda-ax){u}^2\right]^{1}_{-1}
  &=  \int^1_{-1} \frac{d}{dx}\left[ {u}'^2 + (\lambda-ax){u}^2\right] \dx\\
  &= \int^1_{-1}  2{u} {u}'' + (\lambda-ax) 2 {u} {u}' -a  {u}^2 \dx\\
  &= - a \int^1_{-1}   {u}^2 \dx = -a.
\end{align*}
Next, 
\begin{align*}
&\left[ x^2 \left({u}'^2 + (\lambda-ax){u}^2 \right)-2xu u' + u^2
\right]^{1}_{-1}\\
&\hspace{4cm}=\int^{1}_{-1}\frac{d}{dx} \left[ x^2 \left({u}'^2 + (\lambda-ax){u}^2 \right)-2xu u' + u^2\right]\dx\\
&\hspace{4cm}=\int^{1}_{-1} u^2\left[ 4x(\lambda-ax)-a x^2 \right]\dx\\
&\hspace{4cm}= 4\lambda \int^{1}_{-1} x u^2 \dx -5 a  \int^{1}_{-1} x^2 u^2 \dx.
\end{align*}
\end{proof}

\section{Proof of Theorem \ref{constant vs linear}} 

\label{section six}
We begin by estimating the boundary terms of eigenfunctions with linear potentials.   

\begin{lemma} \label{boundary terms}  Let $u_0$ and $u_1$ 
be the first two eigenfunctions of the
Schr\"odinger operator $-\Delta+ax$ with $a>0$,  and Robin boundary conditions~\eqref{main boundary} with $\alpha\ge-1$.  
 Then
$$u_1(1)^2-   u_0(1)^2- u_1(-1)^2+u_0(-1)^2>0.$$
\end{lemma}
\begin{proof}
We apply \eqref{eq:g is 1} 
to
 $u_1$  and $u_0$ in turn:
\begin{equation}  \label{g1 for u1}
{u_1}(1)^2\left( \alpha^2 +\lambda_1 -a\right) -u_1(-1)^2 \left(\alpha^2+\lambda_1+a\right)
= -a
\end{equation}
\begin{equation}  \label{g1 for u0}
{u_0}(1)^2\left( \alpha^2 +\lambda_0 -a\right) -u_0(-1)^2 \left(\alpha^2+\lambda_0+a\right)
= -a.  
\end{equation}
Adding  $2a{u_1}(1)^2$  to both sides of \eqref{g1 for u1} allows it to be rearranged as 
\begin{align*}
\left( \alpha^2 +\lambda_1 + a\right)\left({u_1}(1)^2 -u_1(-1)^2 \right)
&= -a + 2a{u_1}(1)^2\\
&> -a + 2a{u_0}(1)^2
\end{align*}
where in the last line we have used that $u_1(1)^2-u_0^2(1)> 0$, as in \eqref{two zeroes}.   
Similarly we can use \eqref{g1 for u0} to find 
\begin{align*}
-a + 2a{u_0}(1)^2 &= \left( \alpha^2 +\lambda_0 + a\right)\left({u_0}(1)^2 -u_0(-1)^2 \right)\\
&> \left( \alpha^2 +\lambda_1 + a\right)\left({u_0}(1)^2 -u_0(-1)^2 \right)
\end{align*}
where we have used that ${u_0}(1)^2 -u_0(-1)^2<0$,
by Lemma \ref{lemma 5.4}, and $\lambda_0<\lambda_1$.

Combining both these calculations we have
\begin{displaymath}
\left( \alpha^2 +\lambda_1 + a\right)\left({u_1}(1)^2 -u_1(-1)^2 \right)
> \left( \alpha^2 +\lambda_1 + a\right)\left({u_0}(1)^2 -u_0(-1)^2 \right),
\end{displaymath}
and since  $\left( \alpha^2 +\lambda_1 + a\right)>0$
by Lemma \ref{lemma 5.3}, the result follows. 
\end{proof}

In order to prove Theorem \ref{constant vs linear}, we need to show
that the gap for linear potentials $ax$ achieves its minimum for some
finite $a$; later, we'll show that in fact this occurs at $a=0$.

\begin{lemma} \label{coercivity} 
  Let $\Gamma_{\!2}({ax})$ be the fundamental
  gap for the Schr\"odinger operator $-\Delta+ax$, with Robin boundary
  conditions~\eqref{main boundary}.  Then $\Gamma_{\!2}({ax})\to \infty $ as
  $|a|\to \infty$.
\end{lemma}

\begin{proof}
  We only consider the case $a> 0$ since the case $a<0$ proceeds
  similarly. Let $(u_i,\lambda_i^a)$ be an eigenpair for the given
  operator. We rescale the domain using $s=(x+1)a^{-1/3}$, so that
  $w_i(s):=u_i(-1+a^{-1/3}s)$ solves
  \begin{displaymath} 
    \begin{cases}
      -w_i''+sw_i=\hat{\lambda}^a_i w_i &\text{ on } (0,2a^{1/3})\\
      w_i'(0)-\alpha a^{-1/3} w_i(0) &=0, \quad w_i'(2a^{1/3})+\alpha
      a^{-1/3} w_i(2a^{1/3})=0,
    \end{cases}
\end{displaymath}
where
\begin{equation}
  \label{eq:3}
  \hat{\lambda}^a_i=a^{1/3}+a^{-2/3}\lambda_i^a.
\end{equation}
As $a\to \infty$, the eigenpair $(w_i,\hat{\lambda}^a_i)$ approaches
the pair $(v_i,\mu_i)$: in the case that $\alpha\in \R$, 
$v_{i}$ is the $L^{2}(0,+\infty)$-solution to 
\begin{gather} 
  -v_i''+sv_i=\mu_i v_i  \text{ on }(0,\infty) \label{rescaled equation}   
\end{gather}
with $\mu_{i}=\mu_{i}^{N}>0$ and Neumann boundary condition $ v_i'(0)=0$; or if $\alpha=\infty$, $v_{i}$ is the solution to \eqref{rescaled
  equation} with
$\mu_{i}=\mu_{i}^{D}>0$ and Dirichlet boundary condition $v_i(0)=0$.

In both cases the function $v_{i}$ and the eigenvalues
$\mu_{i}^{N}$ and $\mu_{i}^D$ are explicitly known: in the first case,
$v_i(s)=\Ai(s-\mu^N_{i})$, where $\Ai$ is the \emph{Airy function}, the
bounded solution to the ODE $y''(x)=x y(x)$, and the shift $\mu_i^N$
is such that $\Ai'(-\mu_i^N)=0$.  For $i=0$,  $\mu_0^N$ is such
that $v_0>0$ on $(0,\infty)$, and for $i=1$,  $\mu_1^N$ is such
that $v_1$ changes sign once on $(0,\infty)$.  The second case is
similar with $v_i(s)=\Ai(s-\mu^D_{i})$, where $\Ai(-\mu_i^D)=0$.

In either case,  $\hat{\lambda}^a_i \to \mu_i$ as
$a\to+\infty$ and so, by~\eqref{eq:3}, 
\begin{displaymath}
\lambda_i^a= a^{2/3}\mu_i-a + \mathcal{O}(a^{2/3})\qquad\text{as $a\to
  +\infty$.}
\end{displaymath}
Applying this expansion of $\lambda_i^a$ to the fundamental gap
$\Gamma_{\!2}({ax})$ yields
\begin{displaymath}
  \Gamma_{\!2}({ax})=a^{2/3}(\mu_2-\mu_{1})+ \mathcal{O}(a^{2/3})\qquad\text{as $a\to
  +\infty$.}
\end{displaymath}
As $\mu_2-\mu_1>0$, we find $\Gamma_{\!2}({ax})\to \infty$
as $a\to \infty$. 
\end{proof}

\begin{proof}[Proof of Theorem \ref{constant vs linear}]

 Due to  Lemma \ref{coercivity}, the gap for linear potentials $\Gamma({ax})$ achieves a minimum at some $a\in\R$.  
  Define a family of potentials  $V^{t}:=  t a x$.    Using \eqref{gap dot}, \[ \frac{d}{d t}(\lambda^{V^t}_1-\lambda^{V^{t}}_0)= a\int x(u_1^2-u_0^2)\dx,\]
where here and in the remainder of this section $u_0$ and $u_1$ are 
eigenfunctions for the problem \eqref{main}--\eqref{main boundary} with linear potential $V^{t}= ta x$.

Suppose, in order to obtain a contradiction, that the critical point of the gap occurs at some $a>0$.   This implies that 
\begin{equation}
\int x(u_1^2-u_0^2) \dx=0.\label{assumption} \end{equation}

Next, under Robin or Neumann boundary conditions, identity \eqref{eq:g is 1} becomes
\begin{align}  
(\lambda-a+\alpha^2)u(1)^2-(\lambda+a+\alpha^2)u(-1)^2
 &= -a;   \label{eq:g is 1 2}
\end{align}
and identity \eqref{eq:g is xx} becomes
\begin{align*}
u(1)^2\left[ \alpha^2+ \lambda -a+1+2\alpha\right] & -u(-1)^2\left[ \alpha^2 + \lambda + a +1+2\alpha\right] \\&\qquad=4\lambda\int_{-1}^1 x u^2\dx -5a\int_{-1}^1 x^2 u^2\dx,\end{align*}
subtracting  \eqref{eq:g is 1 2} from this results in 
\[
u(1)^2\left[1+2\alpha\right] -u(-1)^2\left[ 1+2\alpha\right] -a =4\lambda\int_{-1}^1 x u^2\dx  -5a\int_{-1}^1 x^2 u^2\dx.\]
Applying this to both $u_0$ and $u_1$, and subtracting, we find
\begin{align}    \label{equation 10}
 \left[1+2\alpha\right] & \left[u_1(1)^2-   u_0(1)^2- u_1(-1)^2+u_0(-1)^2\right]  \\
&=4(\lambda_1-\lambda_0) \int_{-1}^1 x (u_1^2-u_0^2)\dx  -5a\int_{-1}^1 x^2 (u_1^2-u_0^2)\dx.\notag
\end{align}

The assumption $\alpha\ge -\frac12$, and Lemma \ref{boundary terms}, imply that the left hand side is non-negative.
However, we claim that the right hand side is strictly negative.    

The first term of the right hand side is zero, by our assumption \eqref{assumption}.
The final term of \eqref{equation 10}  can be estimated by the same trick we used in Theorem \ref{linear better than convex p}:   let $cx+b$ be the line that intersects $x^2$ at $\xi_{-}$ and $\xi_+$, which are the points where $u_1^2-u_0^2$ changes sign, as in  \eqref{u1-u0 p}.    Then $(x^2-cx-b)(u_1^2-u_0^2)$ 
is 
strictly positive for all $x\not= \xi_{\pm}$.
  Furthermore, $\int u_1^2=\int u_0^2$ and our assumption \eqref{assumption} is that $\int x(u_1^2-u_0^2)=0$, hence
\begin{equation}  \notag -5a \int_{-1}^1 x^2 (u_1^2-u_0^2)\dx=-5a \int_{-1}^1 (x^2-cx-b) (u_1^2-u_0^2)\dx <0.\end{equation}
With these two observations,  \eqref{equation 10} becomes
\begin{equation*}    \label{equation 20}
 \left[1+2\alpha\right]  \left[u_1(1)^2-   u_0(1)^2- u_1(-1)^2+u_0(-1)^2\right]
<0,
\end{equation*}
as claimed.  The contradiction implies our original assumption \eqref{assumption} must be false, and thus the minimum of the gap is not attained for any  potential $ax$ with $a\not=0$.   
 Thus the zero potential, with $a=0$, minimises the gap over all linear potentials.   \end{proof}


\def\cprime{$'$}

 \end{document}